\documentclass{amsart}

\usepackage{sty/preamble}

\newcommand{\bst}{*} % basepoint star
\newcommand{\cDst}{\cD_{\bst}}
\newcommand{\cDstA}{\cDst\mh\A}

\newcommand{\cDpunc}{\cD^\punc}
\newcommand{\cCst}{\cC_{\bst}}
\newcommand{\cCstA}{\cCst\mh\A}

\newcommand{\internalHtyThy}[2]{\bigl(#1\,,\, #2\bigr)} 
\NewDocumentCommand \htythy
{ > { \SplitArgument { 1 } { , } } m }
{ \internalHtyThy #1 }

\newcommand{\cSi}{\cS^{\,i}}

\crefformat{enumi}{(#2#1#3)} % put parentheses around these
\crefmultiformat{enumi}{(#2#1#3)}%
{ and~(#2#1#3)}{,~(#2#1#3)}{ and~(#2#1#3)}
\crefrangeformat{enumi}{(#3#1#4)\crefrangeconjunction(#5#2#6)}

\title{Multifunctorial \texorpdfstring{$K$}{K}-Theory is an Equivalence of Homotopy Theories}

\authorinfoNJDY

\hypersetup{pdfauthor=\authors}

\date{04 October 2022}


\begin{document}

\begin{abstract}
  We show that each of the three $K$-theory multifunctors from small permutative categories to $\mathcal{G}_*$-categories, $\mathcal{G}_*$-simplicial sets, and connective spectra, is an equivalence of homotopy theories.  For each of these $K$-theory multifunctors, we describe an explicit homotopy inverse functor.  As a separate application of our general results about pointed diagram categories, we observe that the right-induced homotopy theory of Bohmann-Osorno $\mathcal{E}_*$-categories is equivalent to the homotopy theory of pointed simplicial categories.

\end{abstract}

\maketitle

\tableofcontents
\section{Introduction}\label{sec:intro}

The Segal $K$-theory functor \cite{segal},
\[
  \Kse\cn \permcatsu \to \Sp_{\ge 0}
\]
provides a construction of connective spectra from small permutative categories.
Its domain, $\permcatsu$, is the category of small permutative categories and strictly unital symmetric monoidal functors.
Its codomain, $\Sp_{\ge 0}$, is the category of connective symmetric spectra.

In \cref{eq:summary}, below, $\Kse$ factors as a composite of three functors.
First, the functor we call Segal $J$-theory, $\Jse$, constructs $\Ga$-categories from small permutative categories.
Then, the levelwise nerve, $N_*$, takes $\Ga$-categories to $\Ga$-simplicial sets.
Finally, there is a construction of symmetric spectra from $\Ga$-simplicial sets, denoted $\Kf$.

The composite, $\Kse$, is known to be an equivalence of stable homotopy theories.
Indeed, it follows from work of Segal \cite{segal}, Bousfield-Friedlander \cite{bousfield-friedlander}, Thomason \cite{thomason}, and Mandell \cite{mandell_inverseK} that each of the intermediate constructions is so.

There is a multifunctorial replacement of Segal $K$-theory, developed by 
Elmendorf and Mandell \cite{elmendorf-mandell,elmendorf-mandell-perm}.
Multifunctoriality is a significant improvement, because it means that the Elmendorf-Mandell construction, $\Kem$, preserves algebraic ring and module structures.
Thus, the Elmendorf-Mandell construction provides a source of highly-structured ring and module spectra.

The Elmendorf-Mandell construction also factors as a composite of three functors, shown along the bottom of \cref{eq:summary}.
\begin{equation}\label{eq:summary}
  \begin{tikzpicture}[x=20mm,y=13mm,scale=1.6,vcenter]
    \draw[0cell] %objects
    (0,0) node (a) {\permcatsu}
    (a)+(30:1.1) node (b) {\GaCat}
    (a)+(-30:1.1) node (b') {\GsCat}
    (b)+(1,0) node (c) {\GasSet}
    (b')+(1,0) node (c') {\GssSet}
    (c)+(-30:1.1) node (d) {\Sp_{\ge 0}}
    ;
    \draw[1cell] %arrows
    (a) edge node {\Jse} (b)
    (a) edge['] node {\Jem} (b')
    (b) edge node {N_*} (c)
    (b') edge node {N_*} (c')
    (c) edge node {\Kf} (d)
    (c') edge['] node {\Kg} (d)
    (b) edge[',transform canvas={shift={(-.1,0)}}] node[pos=.4] {\sma^*} (b')
    (b') edge node[',pos=.5] {i^*} (b)
    (b) edge[bend left=25,transform canvas={shift={(.2,0)}}] node {L} node['] {\vdash} (b')

    (c) edge[transform canvas={shift={(.15,0)}}] node[pos=.4] {\sma^*} (c')
    (c') edge[transform canvas={shift={(.05,0)}}] node[pos=.5] {i^*} (c)
    (c) edge[',bend right=25,transform canvas={shift={(-.2,0)}}] node {L} node['] {\dashv} (c')
    ;
    \draw[1cell]
    (a) [rounded corners=10pt] |- ($(b)+(0,.35)$)
    -- node {\Kse} ($(c)+(0,.35)$) -| (d)
    ;
    \draw[1cell]
    (a) [rounded corners=10pt] |- ($(b')+(0,-.35)$)
    -- node {\Kem} ($(c')+(0,-.35)$) -| (d)
    ;
    \draw[2cell]
    (a) +(-5:.55) node[rotate=-120, 2label={below,\Pi^*\!\!}] {\Rightarrow}
    ;
  \end{tikzpicture}
\end{equation}\ \\
One crucial difference is that $\Kem$ does not factor through $\Ga$-categories or $\Ga$-simplicial sets.
Instead, it uses diagrams on a more complicated but well-structured category $\Gskel$, whose objects are tuples of finite pointed sets.
\cref{sec:Ga-cat-Gs-cat} gives a further review of this background.
Our notations and terminology follow that of \cite{cerberusIII}, and we refer the reader there for a more thorough and unified treatment.
The 2-categorical concepts we use can be found in \cite{johnson-yau}.

It follows from \cite[Theorem~1.1]{elmendorf-mandell} that the connective spectra constructed by $\Kem$ are stably equivalent to those constructed by $\Kse$.  This implies that the multifunctorial $\Kem$ is also an equivalence of stable homotopy theories.  We provide more details of this implication in \cref{theorem:Kem-heq} below.

\subsection*{Main Question and Results}

The focus of this article is on the homotopy theory of $\Gstar$-categories.
One can ask whether there is a notion of stable equivalence for $\Gstar$-diagrams equivalent to that of $\Ga$-diagrams.
There are several approaches, and they lead to not-necessarily equivalent results.
In particular, the notions of stable equivalence created by the functors $\Kg$ and $i^*$ may not be the same.

\cref{corollary:Li-htyeq} shows that the stable equivalences created by $i^*$ are the ones that provide a homotopy theory of $\Gstar$-diagrams equivalent to that of $\Ga$-diagrams.
Thus, the multifunctorial $K$-theory in the title of this paper refers to each of
\begin{itemize}
\item $\Jem\cn \permcatsu \to \GsCat$,
\item $N_* \Jem \cn \permcatsu \to \GssSet$, and
\item $\Kem = \Kg N_* \Jem\cn \permcatsu \to \Sp_{\ge 0}$.
\end{itemize}
\cref{theorem:Jem-sma-N-hty-eq,theorem:Kem-heq} show that each of these is an equivalence of homotopy theories, with stable equivalences created by the vertical functors $i^*$ in \cref{eq:summary}.
These follow as applications of a general construction of right-induced homotopy theory, given in \cref{sec:right-induced}.
There, general conditions are developed for a functor $i$ such that the induced $i^*$ has a left adjoint $L$ and each is an equivalence of homotopy theories.

Thus, each arrow in \cref{eq:summary}, \emph{except for $\Kg$}, is an equivalence of stable homotopy theories.
In the triangle at left,
\[
  \Pi^* \cn \sma^* \Jse \to \Jem
\]
is a natural transformation that induces componentwise stable equivalences; see \cref{theorem:KgNPistar}.
In the middle square, 
\[
  N_* \sma^* = \sma^* N_*,
  \andspace
  i^* N_* = N_* i^*
\]
by associativity of functor composition.
The two adjunctions $L \dashv i^*$ follow from \cref{proposition:Li-adj}.
The triangle at right involving $\sma^*$ commutes \cite[9.3.16]{cerberusIII}:
\[
  \Kg \sma^* = \Kf.
\]

At the end of \cref{sec:applications} we briefly describe a separate application to the Bohmann-Osorno $K$-theory construction from \cite{bohmann_osorno}.  This application is independent from the rest of the work here, but follows from the same general results in \cref{sec:right-induced}.
\cref{corollary:BO-htyeq} shows that the right-induced homotopy theory of Bohmann-Osorno $\cE_*$-categories is equivalent to the homotopy theory of pointed simplicial categories.

In \cref{sec:hinv} we describe explicit homotopy inverses of the $K$-theory multifunctors $\Jem$, $N_* \Jem$, and $\Kem$.  The following is a summary table; its details are explained in \cref{sec:applications,sec:hinv}.
\begin{center}
\resizebox{\textwidth}{!}{
{\renewcommand{\arraystretch}{1.4}%
{\setlength{\tabcolsep}{1ex}
\begin{tabular}{|cc|cc|c|}\hline
\multicolumn{2}{|c|}{$K$-theory multifunctors} & \multicolumn{2}{c|}{Homotopy inverses} & Note \\ \hline
$\Jem \cn \permcatsu \to \GsCat$ & \cref{it:Jem-hty-eq} & $\cP i^* \cn \GsCat \to \permcatsu$ & \cref{eq:jeminv} & $i^*$ is not a multifunctor \\ \hline
$N_* \Jem \cn \permcatsu \to \GssSet$ & (\ref{it:Jem-hty-eq}, \ref{it:Nstar-hty-eq}) & $\cP S_* i^* \cn \GssSet \to \permcatsu$ & \cref{eq:njeminv} & $S_*$ is not a multifunctor \\ \hline
$\Kem \cn \permcatsu \to \Sp_{\geq 0}$ & (\ref{theorem:Kem-heq}) & $\cP S_* \bA \cn \Sp_{\geq 0} \to \permcatsu$ & \cref{eq:keminv} & $\bA$ is not a multifunctor \\ \hline
\end{tabular}}}}
\end{center}
\medskip

\begin{remark}[Open Questions About Alternatives]
  There are a number of alternative ways that one might define a class of stable equivalences in $\GssSet$ and $\GsCat$.  Among them are the following:
  \begin{itemize}
  \item Say that a morphism $f$ of $\Gstar$-simplicial sets is a \emph{$\Kf$-stable equivalence} if $\Kf i^*(f)$ is a stable equivalence.
  \item There is a functor $\#\cn \GssSet \to \GasSet$ that is left adjoint to the functor $\sma^*$ induced by the smash product of pointed finite sets.  Say that $f$ is a \emph{$\#$-stable equivalence} if $\#(f)$ is a stable equivalence of $\Ga$-simplicial sets.
  \end{itemize}
  It is an open question whether these notions of stable equivalence result in homotopy theories that are equivalent to each other, or whether either is equivalent to that of the \emph{$i^*$-stable equivalences} defined in \cref{sec:applications}.  Our results hold for the $i^*$-stable equivalences, created by $i^*$. 
\end{remark}

\subsection*{Acknowledgment}
We thank the referee for several helpful suggestions.

\section{\texorpdfstring{$\Gamma$}{Gamma}-Categories and \texorpdfstring{$\Gstar$}{Gstar}-Categories}
\label{sec:Ga-cat-Gs-cat}

In this section we recall the basic definitions of $\Ga$-categories and $\Gstar$-categories from \cite{segal} and \cite{elmendorf-mandell}, respectively.
A unified treatment is given in \cite[Chapters~8~and~9]{cerberusIII}.

\subsection*{$\Ga$-Diagrams}

We begin by recalling the indexing category $\Fskel$ for $\Ga$-diagrams.

\begin{definition}\label{definition:zero}
  A \emph{zero object} in a category $\C$ is an object $*$ that is both initial and terminal.
  Given a zero object $*$ in $\C$, a \emph{zero morphism} in $\C$ is a morphism that factors through $*$.
  A nonzero morphism is a morphism that does not factor through $*$.
  For objects $a$ and $b$ in $\C$, we write $\C^\punc(a,b)$ for the set of nonzero morphisms $a \to b$.
\end{definition}

\begin{definition}\label{definition:Fskel}
  Let $\Fskel$ denote the category whose objects are pointed finite sets $\ord{n} = \{0,\ldots,n\}$ for natural numbers $n \ge 0$ and whose morphisms are maps of sets that preserve the basepoint element $0$.
  The pointed finite set $\ord{0}$ is an initial and terminal object for $\Fskel$.

  The lexicographic order of smash products gives an isomorphism
  \[
    \ord{m} \sma \ord{n} \iso \ord{mn}
  \]
  making $(\Fskel,\sma,\ord{1})$ a small permutative category.  The symmetry isomorphism is induced by that of the smash product of pointed finite sets.  See, e.g., \cite[8.1.6]{cerberusIII} for further explanation of this structure.
\end{definition}

\begin{definition}
  Let $\Cat$ denote the category of small categories, and let $\boldone$ denote a chosen terminal category, with only a single object $*$ and only a single morphism $1_*$.
\end{definition}

\begin{definition}\label{definition:Ga-cat}
  A \emph{$\Ga$-category} is a basepoint-preserving functor
  \[
    (\Fskel,\ord{0}) \to (\Cat,\boldone).
  \]
  Morphisms of $\Ga$-categories are natural transformations.  Note that the basepoint component of a $\Ga$-category morphism is necessarily the identity, because $\boldone$ is terminal in $\Cat$.  The category of $\Ga$-categories and morphisms is denoted $\Gacat$.

  Likewise, a \emph{$\Ga$-simplicial set} is a basepoint-preserving functor
  \[
    (\Fskel,\ord{0}) \to (\sSet,*)
  \]
  and a morphism of such is given by a natural transformation.
  The category of $\Ga$-simplicial sets is denoted $\Gasset$. 
\end{definition}

\subsection*{$\Gstar$-Diagrams}

Next we recall the indexing category $\Gskel$ for $\Gstar$-diagrams.

\begin{definition}\label{definition:ufs}
  Let $\Inj$ denote the category whose objects are \index{unpointed finite sets}\index{finite sets!unpointed}\emph{unpointed finite sets}
  \[
  \ufs{p} = 
  \begin{cases}
    \{1, \ldots, p\} & \ifspace p > 0,\\
    \varnothing & \ifspace p = 0,
  \end{cases}
  \]
  for each natural number $p \ge 0$, and whose morphisms are \index{injection}injections
  \[
  f\cn \ufs{q} \hookrightarrow \ufs{p}.
  \]

  For each such injection $f$, define a functor
  \[
  f_*\cn \Fskel^q \to \Fskel^p
  \]
  called the \emph{reindexing injection} as follows.
  Suppose given $q$-tuples of pointed finite sets or pointed functions
  \[
    \ordtu{n} = (\ord{n}_1,\ldots,\ord{n}_q)
    \orspace
    \ang{\psi} = (\psi_1, \ldots, \psi_q)
    \in \Fskel^q,
  \]
  respectively.
  Define
  $f_*\ordtu{n}$ to be the $p$-tuple whose $j$th entry is $\ord{n}_{f^\inv(j)}$
  and
  $f_*\ang{\psi}$ to be the $p$-tuple whose $j$th entry is $\psi_{f^\inv(j)}$,
  where
  \[
    \ord{n}_\varnothing = \ord{1} \andspace
    \psi_{\varnothing} = 1_{\ord{1}}.\dqed
  \]
\end{definition}

\begin{definition}\label{definition:Fskel-roundsma}
  For each unpointed finite set $\ufs{q}$ with $q > 0$, define
  \[
  \Fskel^{(q)} = \Fskel^{\sma q},
  \]
  the $q$-fold smash power of pointed categories, where $\Fskel$ has
  basepoint $\ord{0}$.  We denote the objects and morphisms of $\Fskel^{(q)}$ as $q$-tuples
  \[
    \ordtu{n} = (\ord{n}_1,\ldots,\ord{n}_q)
    \andspace
    \ang{\psi} = (\psi_1,\ldots,\psi_q),
  \]
  respectively, where $\ordtu{n}$ is identified with the basepoint of $\Fskel^{(q)}$ if any $\ord{n}_i = \ord{0}$.
  A tuple $\ang{\psi}$ is a zero morphism if any $\psi_i$ factors through $\ord{0}$ in $\Fskel$.
  
  For $q = 0$, define
  \[
  \Ob\,\Fskel^{(0)} = \{*, \ang{}\}
  \]
  where $*$ is the basepoint object and $\ang{}$ is the empty tuple.
  Define the morphisms of $\Fskel^{(0)}$ such that $*$ is both initial and terminal and the only nonzero morphism is the identity morphism of $\ang{}$.
\end{definition}

\begin{definition}\label{definition:Gstar}
  We define a small pointed category $\Gskel$ as follows.
  \begin{description}
  \item[Objects] The set of objects is the wedge of pointed sets
    \[
    \Ob\,\Gskel = \bigvee_{q \ge 0} \Ob \big(\Fskel^{(q)}\big).
    \]

  \item[Morphisms] The basepoint $\starg$ is both initial and terminal
    in $\Gskel$.  The set of morphisms from a $q$-tuple $\ordtu{n}$ to
    a $p$-tuple $\ordtu{m}$ is
    \begin{align}
    \Gskel(\ordtu{n},\ordtu{m}) \label{eq:Gskelnm}
    & = \bigvee_{f\in \Inj(\ufs{q}\,,\,\ufs{p})}\;\Big(\,
    \Fskel^{(p)}\big(f_*\ordtu{n}, \ordtu{m}\big)
    \,\Big)\\
    & = \bigvee_{f\in \Inj(\ufs{q}\,,\,\ufs{p})}\;\Big(\,
    \Opsma_{j=1}^p \Fskel\big(\ord{n}_{f^\inv(j)}, \ord{m}_j\big)
    \,\Big). \nonumber
    \end{align}

    In \eqref{eq:Gskelnm} for $p > 0$ we denote a morphism by a pair
    $(f,\ang{\psi})$, where
    \[
    f\cn \ufs{q} \hookrightarrow \ufs{p}  \inspace \Inj
    \]
    and
    \[
    \ang{\psi}\cn f_*\ordtu{n} \to \ordtu{m}
    \]
    is a morphism in $\Fskel^{(p)}$.
    A morphism $( f, \ang{\psi} )$ is identified with the zero morphism in $\Gskel(\ordtu{n}, \ordtu{m})$ if there exists a component morphism
    \[
    \psi_j \cn \ord{n}_{f^\inv (j)} \to \ord{m}_j 
    \]
    that is a zero morphism in $\Fskel$, factoring through $\ord{0}$.

  \item[Identities]
    The identity on a $q$-tuple $\ordtu{n}$ is given by
    the pair $(1_{\ufs{q}},1_{\ordtu{n}})$.

  \item[Composition] The composite of
    morphisms
    \[
    \ordtu{n} \fto{(f,\ang{\psi})} \ordtu{m} \fto{(g,\ang{\phi})} \ang{\ord{\ell}}
    \]
    is given by the pair $(gf,\ang{\phi}\circ g_*\ang{\psi})$.
  \end{description}
  This finishes the definition of $\Gskel$.  The composition in
  $\Gskel$ is associative and unital since $(gf)_* = g_* \circ f_*$
  for composable injections $f$ and $g$.  Moreover, we note the
  following.
  \begin{itemize}
  \item We call $q$ the \emph{length} of a $q$-tuple $\ordtu{n}$.
  \item Each morphism set $\Gskel(\ordtu{n},\ordtu{m})$ is a pointed
    set with basepoint the zero morphism.
  \item For readability, we sometimes use a semicolon and write
    $\Gskel\mmap{\ordtu{m};\ordtu{n}}$ for the set of morphisms in
    $\Gskel$ from $\ordtu{n}$ to $\ordtu{m}$.\dqed
  \end{itemize}
\end{definition}

We define $\Gstar$-categories and $\Gstar$-simplicial sets as pointed diagrams, similar to the corresponding versions for $\Fskel$ in \cref{definition:Ga-cat}.
\begin{definition}\label{definition:Gs-cat}
  The category of $\Gstar$-categories is denoted $\GsCat$ and consists of basepoint-preserving functors
  \[
    (\Gskel,*) \to (\Cat,\boldone)
  \]
  together with natural transformations between them.  Likewise, the category of $\Gstar$-simplicial sets is denoted $\GssSet$ and consists of basepoint-preserving functors
  \[
    (\Gskel,*) \to (\sSet,*)
  \]
  together with natural transformations between them.
\end{definition}

\subsection*{Relating $\Fskel$ and $\Gskel$}

The category $\Gskel$ is a permutative category---that is, a \emph{strict} symmetric monoidal category---in which the monoidal product $\oplus$ is given by concatenation on objects and similarly for morphisms.  The monoidal unit is the empty sequence $\ang{}$.  We will use the following functors to compare $\Ga$-diagrams and $\Gstar$-diagrams.

\begin{definition}[{\cite[9.1.15]{cerberusIII}}]\label{definition:sma-functor}
  The smash product of pointed finite sets defines a strict symmetric monoidal basepoint-preserving functor
  \[
  \sma\cn (\Gskel,\oplus,\ang{},\starg) \to (\Fskel,\sma,\ord{1},\ord{0})
  \]
  given on objects by the following assignments:
  \begin{align*}
    \sma \ordtu{n} & = \ord{n_1 \cdots n_q} \forspace q > 0,\\
    \sma \ang{} & = \ord{1}, \andspace\\
    \sma \starg & = \ord{0}.
  \end{align*}
  The value of $\sma$ on a morphism $(f,\ang{\psi})$ is given by the smash product $\sma_j \psi_j$ together with lexicographic ordering and a permutation of factors induced by $f$.
  See \cite[9.1.15]{cerberusIII} for further details.
  There are functors
  \[
    \sma^* \cn \GaCat \to \GsCat \andspace \sma^*\cn \GasSet \to \GssSet
  \]
  induced by precomposition with $\sma$ on objects and whiskering with $\sma$ on morphisms.
\end{definition}

\begin{definition}\label{definition:i-incl}
  The \emph{length-one inclusion}
  \begin{equation}\label{eq:i-incl}
    i\cn \Fskel \to \Gskel
  \end{equation}
  sends each pointed finite set $\ord{n}$ to the length-one tuple $(\ord{n})$ and each morphism $\psi$ in $\Fskel$ to the pair $(1_{\ufs{1}},(\psi))$.
  There are functors
  \begin{equation}\label{eq:istar}
    i^* \cn \GsCat \to \GaCat \andspace i^* \cn \GssSet \to \GasSet
  \end{equation} 
  induced by precomposition with $i$ on objects and whiskering with $i$ on morphisms.
\end{definition}

The functor $i$ is fully faithful; this will be used in \cref{corollary:Li-htyeq} below.  Moreover, there is an equality
\[\sma \circ i = 1_{\Fskel} \cn \Fskel \to \Gskel \to \Fskel.\]
This will be used in the proof of \cref{it:smastar-hty-eq,it:smastarsset-hty-eq} below.  

We also note that, unlike $\sma$, the functor $i \cn \Fskel \to \Gskel$ is neither monoidal nor oplax monoidal.  In more detail, there are, in general, no morphisms
\[i(\ord{m}) \oplus i(\ord{n}) = (\ord{m}, \ord{n}) \to (\ord{mn}) = i(\ord{m} \sma \ord{n})\]
in $\Gskel$, so $i$ does not admit any monoidal constraint.  There are also no morphisms
\[i(\ord{1}) = (\ord{1}) \to \ang{}\]
in $\Gskel$, so $i$ does not admit any oplax unit constraint.

\section{Right-Induced Equivalences of Homotopy Theories}
\label{sec:right-induced}

In this section we describe the general framework that we will apply in \cref{corollary:Li-htyeq} to show that the functors
\[
  i^* \cn \GsCat \to \GaCat \andspace i^* \cn \GssSet \to \GasSet
\]
of \cref{eq:istar} induce equivalences of homotopy theories.  Equivalences of homotopy theories are phrased in terms of weak equivalences in the complete Segal space model structure on bisimplicial sets, which we recall in \cref{theorem:css-fibrant,def:eq-htpy} below.  

\subsection*{Motivation for Our Framework}

Before we recall the relevant definitions of equivalences of homotopy theories, let us briefly discuss its relationship with Quillen model categories \cite{hovey,hirschhorn,quillen-homotopical-algebra}.  Consider the diagram \cref{eq:summary}.  Several well-known model category structures on symmetric spectra are discussed in \cite{hss,mmss}.  Moreover, the categories $\GaCat$ and $\GasSet$ admit Reedy model structures, which are constructed using the chosen model structures on $\Cat$ and $\sSet$.  However, the category $\permcatsu$ is not known to admit a model category structure that is compatible with the one on $\GaCat$.  

Moreover, the $K$-theory functors $\Jse$, $\Kse$, $\Jem$, $\Kem$, and $\Kg$ are not known to admit adjoint functors.  Thus we cannot directly apply model category theory to these $K$-theory functors. 
On the other hand, each Quillen equivalence induces an equivalence of homotopy theories in the sense of \cref{theorem:css-fibrant,def:eq-htpy} below.
We discuss this point further in \cref{rk:DK} after the following definitions.

\subsection*{Equivalences of Homotopy Theories}

We briefly recall the following background concepts for homotopy theories and equivalences thereof.
Our work below will use only a special case from \cite[2.9]{gjo1}, stated in \cref{gjo29} below.
We refer the reader to \cite{johnson-yau-permmult} for details of the following.
See \cite{dwyer-kan,hirschhorn,toen-axiomatisation,barwick-kan} for further development of the theory.
\begin{definition}\ 
\begin{enumerate}
\item A \emph{relative category} is a pair $\htythy{\C,\cW}$ consisting of a category $\C$ and a subcategory $\cW$ containing all of the objects of $\C$.
  A \emph{relative functor} from $\htythy{\C,\cW}$ to $\htythy{\C',\cW'}$ is a functor from $\C$ to  $\C'$ that sends morphisms of $\cW$ to those of $\cW'$.

\item We say that a subcategory $\cW \subset \C$ satisfies the \emph{2-out-of-3} property if, for each composable pair of morphisms $f$ and $g$ in $\C$, whenever any two of $f$, $g$, and $gf$ are in $\cW$, then so is the third.

\item We say that a functor $F \cn \C \to \C'$ \emph{creates morphisms in $\cW$} if $\cW = F^\inv(\cW')$.
  
\item For each $n \ge 0$, the \emph{relative simplex category} $\htythy{\C,\cW}^{\De[n]}$ has objects given by functors $\Delta[n] \to \C$.
  The morphisms are natural transformations whose components are morphisms in $\cW$.

\item Each relative category $\htythy{\C,\cW}$ has a bisimplicial \emph{classification diagram} $\Ncl\htythy{\C,\cW}$ formed by taking levelwise nerve of the relative simplex categories $\htythy{\C,\cW}^{\De[n]}$ for $n \ge 0$.\dqed
\end{enumerate}
\end{definition}

The category of bisimplicial sets has a Reedy model structure \cite{bousfield-friedlander,goerss-jardine}.  Equivalences of homotopy theories are defined using a different model structure on bisimplicial sets, which we recall next.

\begin{definition}\label{definition:css}
A bisimplicial set $Y$ is called a \emph{complete Segal space} if it satisfies the following three conditions.
  \begin{itemize}
  \item $Y$ is fibrant in the Reedy model structure on bisimplicial sets.
  \item For each $n \ge 2$ the Segal morphism
    \[
    Y_n \fto{\sim} Y_1 \times_{Y_0} \cdots \times_{Y_0} Y_1
    \]
    is a weak equivalence of simplicial sets.
  \item The morphism
\begin{equation}\label{eq:css-char}
Y_0 \iso \Map(\De[0],Y) \fto{\sim} \Map(\mathbf{2},Y)
\end{equation}
is a weak equivalence of simplicial sets.  Here $\mathbf{2}$ is the discrete nerve of the category consisting of two isomorphic objects.  The morphism in \eqref{eq:css-char} is induced by the unique morphism $\mathbf{2} \to \De[0]$. \dqed
  \end{itemize}
\end{definition}

\begin{theorem}[{\cite[7.2]{rezk-homotopy-theory}}]\label{theorem:css-fibrant}
There is a simplicial closed model structure on the category of bisimplicial sets, called the \emph{complete Segal space model structure}, that is given as a left Bousfield localization of the Reedy model structure in which the fibrant objects are precisely the complete Segal spaces.
\end{theorem}

\begin{definition}\label{def:eq-htpy}
Suppose $\htythy{\C,\cW}$ and $\htythy{\C',\cW'}$ are relative categories.  A relative functor
  \[
    F \cn \htythy{\C,\cW} \to \htythy{\C',\cW'}
  \]
  is an \emph{equivalence of homotopy theories} if the induced morphism $\mathsf{R}(\Ncl F)$ is a weak equivalence in the complete Segal space model structure, where $\mathsf{R}$ denotes fibrant replacement.
\end{definition}

\begin{remark}[Homotopy Categories]\label{rk:DK}
Continuing the discussion near the beginning of this section, let us point out that, in the presence of a Quillen model structure, \cref{def:eq-htpy} yields the same notion of an \emph{equivalence of homotopy categories} as in the model categorical sense.  To make this precise, first note that a relative functor 
\[F \cn (\C,\cW) \to (\C',\cW')\]
induces a functor between the categorical localizations
\begin{equation}\label{Fbar}
\ol{F} \cn \C[\cW^\inv] \to \C'[(\cW')^\inv]
\end{equation}
obtained by formally inverting the morphisms in $\cW$ and $\cW'$, respectively.  Work of Barwick and Kan \cite[1.8]{barwick-kan} shows that a relative functor is an equivalence of homotopy theories in the sense of \cref{def:eq-htpy} if and only if it induces a $DK$-equivalence between hammock localizations in the sense of \cite{dwyer-kan}.  Thus if $F$ is an equivalence of homotopy theories, then it induces 
\begin{itemize}
\item weak equivalences between mapping simplicial sets and
\item equivalences between categories of components.
\end{itemize} 
It follows that, if a relative functor $F$ is an equivalence of homotopy theories, then the induced functor $\ol{F}$ in \cref{Fbar} is an equivalence of categories.
Note that this conclusion holds independently of whether $F$ is assumed to be part of a Quillen adjunction.

On the other hand, suppose $\C$ is a model category with $\cW$ as the class of weak equivalences.   Its homotopy category $\Ho(\C)$ in the model categorical sense has
\begin{itemize}
\item the fibrant-cofibrant objects in $\C$ as objects and
\item homotopy classes of morphisms in $\C$ as morphisms.
\end{itemize} 
By the defining universal property of categorical localizations, there is an equivalence of categories
\[I_\C \cn \Ho(\C) \fto{\sim} \C[\cW^\inv]\]
that is natural in the following sense.
If $F \cn \C \to \C'$ is part of a Quillen adjunction, then the following induced diagram of functors is commutative up to a natural isomorphism.  
\begin{equation}\label{HoF}
\begin{tikzpicture}[vcenter]
\def\v{-1.4}
\draw[0cell]
(0,0) node (a1) {\Ho(\C)}
(a1)+(3,0) node (a2) {\Ho(\C')}
(a1)+(0,\v) node (b1) {\C[\cW^\inv]}
(a2)+(0,\v) node (b2) {\C'[(\cW')^\inv]}
;
\draw[1cell=.9]
(a1) edge node {\Ho(F)} (a2)
(a2) edge node {I_{\C'}} node[swap] {\sim} (b2)
(a1) edge node {\sim} node[swap] {I_\C} (b1)
(b1) edge node {\ol{F}} (b2)
;
\draw[2cell]
node[between=a1 and b2 at .5, rotate=210, 2label={below,\iso}] {\Rightarrow}
;
\end{tikzpicture}
\end{equation}
If, moreover, $F$ is part of a Quillen equivalence, then $\Ho(F)$ is an equivalence of categories.  In this case, the diagram \cref{HoF} and the 2-out-of-3 property imply that $\ol{F}$ is also an equivalence of categories.  Conversely, if $F$ is an equivalence of homotopy theories in the sense of \cref{def:eq-htpy}, then $\ol{F}$ is an equivalence of categories.  Thus \cref{HoF} and the 2-out-of-3 property imply that $\Ho(F)$ is an equivalence of categories.
\end{remark}

Relative functors are closed under composition, but they do not satisfy the 2-out-of-3 property in general.  However, the following observation follows from the 2-out-of-3 property of equivalences of homotopy theories.
\begin{lemma}\label{lemma:rel-2-of-3}
  Suppose $\htythy{\C,\cW}$, $\htythy{\C',\cW'}$, and $\htythy{\C'',\cW''}$ are relative categories, and suppose 
  \[
    F \cn \C \to \C' \andspace F' \cn \C' \to \C''
  \]
  are functors such that the following two statements hold.
  \begin{enumerate}
  \item The functor $F' \cn \htythy{\C',\cW'} \to \htythy{\C'',\cW''}$ creates morphisms in $\cW'$ and is an equivalence of homotopy theories.
  \item The composite $F'F \cn \htythy{\C,\cW} \to \htythy{\C'',\cW''}$ is a relative functor and is an equivalence of homotopy theories.
  \end{enumerate}
  Then $F \cn \htythy{\C,\cW} \to \htythy{\C',\cW'}$ is a relative functor and is an equivalence of homotopy theories.
\end{lemma}

Next we give two results that will be used below to show that certain functors are equivalences of homotopy theories.
Both follow from the observation that a natural transformation between functors induces a simplicial homotopy on nerves.
\begin{lemma}\label{lemma:hteq-transf}
  Suppose
  \[
    F,F' \cn \htythy{\C,\cW} \to \htythy{\C',\cW'}
  \]
  are relative functors and suppose $p \cn F \to F'$ is a natural transformation with each component in $\cW'$.
  Then $F$ is an equivalence of homotopy theories if and only if $F'$ is an equivalence of homotopy theories.
\end{lemma}
\begin{proof}
  Because the components of $p$ are in $\cW'$, it induces, for each $n \ge 0$, a natural transformation
  \[
    p^{\De[n]} \cn F^{\De[n]} \to (F')^{\De[n]}
    \cn \htythy{\C,\cW}^{\De[n]} \to \htythy{\C',\cW'}^{\De[n]}
  \]
  between the induced functors on relative simplex categories.
  Therefore, taking nerves, $N(p^{\De[n]})$ is a simplicial homotopy for each $n \ge 0$.
  Thus,
  \[
    \Ncl(p) \cn \Ncl(F) \to \Ncl(F')
    \cn \Ncl\htythy{\C,\cW} \to \Ncl\htythy{\C',\cW'}
  \]
  is a levelwise simplicial homotopy between morphisms of bisimplicial sets.  The result then follows from functoriality of fibrant replacements and 2-out-of-3 for weak equivalences in Rezk's complete Segal space model structure.
\end{proof}

The previous result will be useful on its own, but is also used in the proof of the following, from \cite[2.8,2.9]{gjo1}.
Its proof is similar to that of \cite[2.12]{johnson-yau-permmult}.
\begin{proposition}[{\cite[2.9]{gjo1}}]\label{gjo29}
  Suppose given relative functors
  \[
    F \cn \htythy{\C,\cW} \lradj \htythy{\C',\cW'} \cn G
  \]
  such that each of the composites $GF$ and $FG$ is connected to the respective identity functor by a zigzag of natural transformations whose components are in $\cW$ and $\cW'$, respectively.
  Then $F$ and $G$ are equivalences of homotopy theories.
\end{proposition}

\subsection*{Adjunction Between Pointed Diagram Categories}

Next we describe a general context where a right adjoint of diagram categories induces an equivalence of homotopy theories.  We will apply this in \cref{sec:applications} to the length-one inclusion \cref{eq:i-incl}
\[
  i \cn \Fskel \to \Gskel
\]
with the codomain categories $(\Cat,\boldone)$ and $(\sSet,*)$.
\begin{definition}
  Suppose $(\A,*)$ is a cocomplete category with a chosen terminal object $*$.
  A \emph{pointed object} in $\A$ is a pair $(x,\iota^x)$ consisting of an object $x$ in $\A$ and a morphism $\iota^x \cn * \to x$ in $\A$.
  A \emph{pointed morphism} $f \cn (x,\iota^x) \to (y,\iota^y)$ between pointed objects is a morphism $f \cn x \to y$ such that $f \iota^x = \iota^y$.
  For pointed objects $(x,\iota^x)$ and $(y,\iota^y)$, the \emph{wedge} $x \wed y$ is the pointed object given by the pushout of $x \leftarrow * \to y$ in $\A$.
  A wedge indexed by an empty indexing set is defined to be the chosen terminal object $*$.
\end{definition}

\begin{convention}\label{convention:iCDA}
  Suppose $\cC$ and $\cD$ are small categories, each with a chosen initial and terminal basepoint object denoted $*$.  Suppose
  \[
    i \cn \cC \to \cD
  \]
  is a basepoint-preserving functor, so $i* = *$.
  Suppose $\A$ is a complete and cocomplete category with chosen terminal object $*$. 
  Let $\cCstA$ and $\cDstA$ denote the categories of terminal-object-preserving functors 
  \[
    (\cC,*) \to (\A,*) \andspace (\cD,*) \to (\A,*),
  \]
  respectively, with morphisms given by natural transformations.
  Then there is a functor
  \[
    i^* \cn \cDstA \to \cCstA
  \]
  given by precomposition with $i$ on objects and whiskering with $i$ on morphisms.
\end{convention}

We assume the context of \cref{convention:iCDA} for the remainder of this section.
\begin{definition}[Functor $L$]\label{definition:L-coend}
  Define a functor
  \[
    L \cn \cCstA \to \cDstA
  \]
  by the following coend construction in $\A$.
  For $X \in \cCstA$ and $t \in \cD$,
  \[
    (LX)t = \int^{n \in \cC} \bigvee_{\cDpunc(i(n), t)} Xn,
  \]
  where, recalling \cref{definition:zero}, $\cDpunc(i(n),t)$ denotes the set of nonzero morphisms in $\cD$ from $i(n)$ to $t$.
  If $\cDpunc(i(n),t)$ is empty (for example, if $n = *$ or if $t = *$), then the displayed wedge above is the chosen terminal object $*$.  This ensures that $(LX)* = *$ in $\A$.
  The definition of $LX$ on morphisms $t \to s$ in $\cD$
  is given by functoriality of the indexing sets $\cDpunc(i(n),t)$ and of the coend.
\end{definition}

\begin{definition}[Unit $\eta$]\label{definition:unit-eta}
  Define a unit natural transformation whose component at $X \in \cCstA$,
  \[
    X \fto{\eta_X} i^* L X,
  \]
  is given by the following composite, for each $p \in \cC$, where the first unlabeled arrow includes to the component at $1_{i(p)}$ and the second is the coend structure morphism for $n = p$.
  \[
    \begin{tikzpicture}[x=20mm,y=10mm,baseline={(c.base)}]
      \draw[0cell] %objects
      (0,0) node (a) {Xp}
      (1,-1) node (b) {\bigvee_{\cDpunc(i(p),i(p))} Xp}
      (2,0) node (c) {\bigl( i^*L X \bigr)p}
      ;
      \draw[1cell] %arrows
      (a) edge node {} (b)
      (b) edge node {} (c)
      (a) edge node {\eta_{X,p}} (c)
      ;
    \end{tikzpicture}
    = 
    \int^{n \in \C} \bigvee_{\cDpunc(i(n),i(p))} Xn
  \]

  Naturality of $\eta_X$ with respect to morphisms $h\cn p \to p'$ in $\cC$ follows from the coend condition that the following square commutes, where the two arrows at left are induced by $i(h)$ and $Xh$, and the two arrows at right are the structure morphisms for $n=p$ and $n=p'$.
  \[
    \begin{tikzpicture}[x=40mm,y=20mm,baseline={(d.base)}]
      \draw[0cell] %objects
      (0,0) node (a) {\bigvee_{\cDpunc(i(p'),i(p'))} Xp}
      (a)+(45:1) node (b) {\bigvee_{\cDpunc(i(p),i(p'))} Xp}
      (a)+(-45:1) node (c) {\bigvee_{\cDpunc(i(p'),i(p'))} Xp'}
      (b)+(-45:1) node (d) {\bigl( i^*L X \bigr)p'}
      ;
      \draw[1cell] %arrows
      (a) edge node {} (b)
      (b) edge node {} (d)
      (a) edge node {} (c)
      (c) edge node {} (d)
      ;
    \end{tikzpicture}
    = 
    \int^{n \in \C} \bigvee_{\cDpunc(i(n),i(p'))} Xn
  \]
  Naturality of $\eta$ with respect to morphisms $f\cn X \to X'$ in $\cCstA$ follows from naturality of $f$ with respect to morphisms in $\cC$.
\end{definition}

\begin{definition}[Counit $\epz$]\label{definition:counit-epz}
  Define a counit natural transformation whose component at $Y \in \cDstA$,
  \[
    L i^* Y \fto{\epz_Y} Y
  \]
  is given as follows.  In the diagram below, $\theta \cn i(n) \to t$ is a nonzero morphism in $\cD$, for $n \in \cC$ and $t \in \cD$.
  The top vertical arrow is given by the identity on the summand at $\theta$.
  Then the unlabeled dashed arrow is the universal morphism out of the wedge, induced by $Y\theta$ for each $\theta \in \cDpunc(i(n),t)$.
  \begin{equation}\label{eq:epz-theta}
    \begin{tikzpicture}[x=50mm,y=16mm,vcenter]
      \draw[0cell] %objects
      (0,0) node (a) {\int^{n \in \cC} \bigvee_{\cDpunc(i(n),t)} Yi(n)}
      (a.west) +(-9mm,0) node (a2) {\bigl( L i^* Y \bigr)t }
      (0,1) node (b) {\bigvee_{\cDpunc(i(n),t)} Yi(n)}
      (0,2) node (c) {Yi(n)}
      (1,0) node (d) {Yt}
      ;
      \draw[1cell] %arrows
      (c) edge node {} (b)
      (b) edge node {} (a)
      (c) edge node[pos=.4] {Y\theta} (d)
      (a) edge[dashed] node[pos=.4] {\epz_{Y, t}} (d)
      (b) edge[dashed] node {} (d)
      (a) edge[equal] node {} (a2)
      ;
    \end{tikzpicture}
  \end{equation}
  The bottom vertical arrow in the above diagram is the structure morphism for the coend, and $\epz_{Y,t}$ is the induced universal morphism.  Naturality of $\epz_Y$ with respect to morphisms $t \to t'$ in $\cD$ follows from naturality of the universal morphisms above.  Naturality of $\epz$ with respect to morphisms $g \cn Y \to Y'$ in $\cDstA$ follows from naturality of $g$ with respect to morphisms in $\cD$ and uniqueness of the induced morphism
  \[
    \int^{n \in \C} \bigvee_{\cDpunc(i(n),t)} Yi(n) \to Y' t
  \]
  given by either $\epz_{Y,t}$ and $g_t$ or by $\bigvee_\theta g_{i(n)}$ and $\epz_{Y',t}$.
\end{definition}

\begin{proposition}\label{proposition:Li-adj}
  In the context of \cref{convention:iCDA}, there is an adjunction of categories
  \[
    L \cn \cCstA \lradj \cDstA \cn i^*.
  \]
\end{proposition}
\begin{proof}
  The unit and counit are those of \cref{definition:unit-eta,definition:counit-epz}, respectively.
  The two triangle identities are verified componentwise as follows.

  First, for $X \in \cCstA$ and $t \in \cD$, consider the composite
  \begin{equation}\label{eq:epz-eta}
    \begin{tikzpicture}[x=65mm,y=20mm]
      \draw[0cell] %objects
      (0,0) node (a) {
        \int^{p \in \cC} \bigvee_{\cDpunc(i(p),t)} Xp
      }
      (a.west)++(-8mm,0) node (a') {\bigl( LX  \bigr)t}
      (1,0) node (b) {
        \int^{p \in \cC} \bigvee_{\cDpunc(i(p),t)}
        \left[ \int^{n \in \cC} \bigvee_{\cDpunc(i(n),i(p))} Xn \right]
      }
      (1,-1) node (c) {
        \int^{n \in \cC} \bigvee_{\cDpunc(i(n),t)} Xn
      }
      (c.east)++(8mm,0) node (c') {\bigl( LX  \bigr)t.}
      ;
      \draw[1cell] %arrows
      (a) edge node {L(\eta_{X,p})_t} (b)
      (b) edge node {\epz_{LX,t}} (c)
      (a') edge[equal] node {} (a)
      (c) edge[equal] node {} (c')
      ;
    \end{tikzpicture}
  \end{equation}
  The above composite is the unique universal morphism out of the coend $\bigl( LX  \bigr)t $ induced by the following composite, where we use the notation change $m = p$ on the left and $m = n$ on the right:
  \[
    \bigvee_{\cDpunc(i(m),t)} Xm
    \to
    \bigvee_{\cDpunc(i(m),t)} \bigvee_{\cDpunc(i(m),i(m))} Xm
    \to
    \bigvee_{\cDpunc(i(m),t)} Xm.
  \]
  In the above composite, the first morphism is induced by $\bigvee_{\cDpunc(i(m),t)} 1_{i(m)}$ on indexing sets, and the second by composition on indexing sets.
  Thus the above composite is the identity, and therefore so is \cref{eq:epz-eta}.

  Second, for $Y \in \cDstA$ and $p \in \cC$, consider the composite
  \begin{equation}\label{eq:eta-epz}
    \begin{tikzpicture}[x=60mm,y=20mm,vcenter]
      \draw[0cell] %objects
      (0,0) node (a) {\bigl( i^*Y  \bigr)p}
      (.4,-.6) node (a2) {\bigvee_{\cDpunc(i(p),i(p))} Yi(p)}
      (1,0) node (b) {\int^{n \in \cC} \bigvee_{\cDpunc(i(n),i(p))} Yi(n)}
      (1,-1) node (c) {Yi(p)}
      (c.east)++(8mm,0) node (c') {\bigl( i^*Y  \bigr)p.}
      ;
      \draw[1cell] %arrows
      (a) edge node {\eta_{i^* Y, p}} (b)
      (b) edge node {(i^*\epz_{Y})_p = \epz_{i^*Y,i(p)}} (c)
      (c) edge[equal] node {} (c')
      (a) edge node {} (a2)
      (a2) edge node {} (b)
      ;
    \end{tikzpicture}
  \end{equation}
  The above composite is the identity because the first morphism is induced by including to the component at $1_{i(p)}$ and the composite to $Yi(p)$ is that defining $\epz_{i^*Y, i(p)}$ in \cref{eq:epz-theta} with $\theta = 1_{i(p)}$.
\end{proof}

\begin{lemma}\label{lemma:eta-iso}
  Suppose, in the context of \cref{convention:iCDA}, that $i$ is fully faithful.  Then the unit $\eta$ of the adjunction $L \dashv i^*$ is a natural isomorphism.
\end{lemma}
\begin{proof}
  For $X$ in $\cCstA$ and $p \in \cC$, the Yoneda Density Theorem \cite[III.3.7.8]{cerberusIII} gives an isomorphism
  \[
    \int^{n \in \cC} \bigvee_{\cC^\punc(n,p)} Xn \fto{\iso} Xp
  \]
  that is natural with respect to morphisms $p \to p'$ in $\cC$.
  The assumption that $i$ is fully faithful implies there is an isomorphism
  \[
    \cDpunc(i(n),i(p)) \iso \cC^\punc(n,p) \forspace n,p \in \cC.
  \]
  The above isomorphism, applied to indexing sets, gives the middle isomorphism below:
  \begin{equation}\label{eq:eta-iso}
    Xp \fto{\eta_{X,p}}
    \int^{n \in \cC} \bigvee_{\cDpunc(i(n),i(p))} Xn \iso
    \int^{n \in \cC} \bigvee_{\cC^\punc(n,p)} Xn \fto{\iso} Xp.
  \end{equation}
  The above composite is the identity on $Xp$ because $\eta_{X,p}$ includes $Xp$ at the component indexed by $1_{i(p)}$ and the Yoneda Density isomorphism is induced by applying $X$ to the morphisms of $\cC^\punc(n,p)$.
  Therefore $\eta_{X,p}$ is an isomorphism with inverse given by the other morphisms of \cref{eq:eta-iso}.
\end{proof}

\subsection*{Equivalences of Homotopy Theories Between Pointed Diagram Categories}

\begin{definition}\label{definition:rind-hty-thy}
  Suppose, in the context of \cref{convention:iCDA}, that $\cS \subset \cCstA$ is a subcategory containing all objects of $\cCstA$, so  $\htythy{\cCstA,\cS}$ is a relative category.
  The \emph{right-induced subcategory} of $\cDstA$ is the subcategory 
  \[
    \cSi = (i^*)^\inv(\cS) \subset \cDstA
  \]
  consisting of morphisms $f$ such that $i^*(f) \in \cS$.
  Thus, $\htythy{\cDstA,\cSi}$ is a relative category and
  \[
    i^* \cn \htythy{\cDstA,\cSi} \to \htythy{\cCstA,\cS}
  \]
  is a relative functor with the morphisms of $\cSi$ created by $i^*$.
  If $\cS$ contains all isomorphisms, then so does $\cSi$.
  If $\cS$ satisfies the 2-out-of-3 property, then so does $\cSi$.
\end{definition}

\begin{theorem}\label{theorem:Li-adj-he}
  In the context of \cref{convention:iCDA}, suppose $\cS \subset \cCstA$ is a subcategory containing all objects and suppose, moreover, that the following three conditions hold.
  \begin{itemize}
  \item The functor $i\cn \cC \to \cD$ is fully faithful.
  \item The subcategory $\cS$ contains all isomorphisms.
  \item The subcategory $\cS$ has the 2-out-of-3 property.
  \end{itemize}
  Let $\cSi = (i^*)^\inv(\cS)$ be the right-induced subcategory of $\cDstA$.  Then the adjunction $L \dashv i^*$ from \cref{definition:L-coend} induces an equivalence of homotopy theories
  \[
    L \cn \htythy{\cCstA, \cS} \lrsimadj \htythy{\cDstA, \cSi} \cn i^*.
  \]
\end{theorem}
\begin{proof}
  Since $i$ is fully faithful, each component of the unit $\eta$ is an isomorphism by \cref{lemma:eta-iso}.
  Therefore each $\eta_X$ is a morphism of $\cS$.
  This implies that $L$ is a relative functor, by naturality of $\eta$ and the 2-out-of-3 property for $\cS$.
  For each component $\epz_Y$, the right triangle identity
  \[
    (i^*\epz_Y) \circ \eta_{i^*Y} = 1_{i^*Y}
  \]
  from \cref{eq:eta-epz}, together with the 2-out-of-3 property for $\cS$, implies that $i^*\epz_Y$ is a morphism of $\cS$ and hence $\epz_Y$ is a morphism of $\cSi = (i^*)^\inv(\cS)$.
  Since the components of $\eta$, respectively $\epz$, are morphisms of $\cS$, respectively $\cSi$, the result then follows from \cref{gjo29}.
\end{proof}

\section{Applications to \texorpdfstring{$K$}{K}-Theory Multifunctors}
\label{sec:applications}

This section contains our main applications of \cref{theorem:Li-adj-he}, with $i \cn \Fskel \to \Gskel$ the inclusion of length-one tuples from \cref{definition:i-incl}.
At the end of the section, we give a separate application to Bohmann-Osorno $\cE_*$-categories.

\begin{definition}
  Let $\cS \subset \Sp_{\ge 0}$ denote the subcategory of \emph{stable equivalences} of connective spectra.  Then, in each of the following, let $\cS$ denote the subcategory of morphisms created by the indicated functor:
  \[
    \htythy{\permcatsu,\cS} \fto{\Jse}
    \htythy{\GaCat,\cS} \fto{N_*}
    \htythy{\GasSet,\cS} \fto{\Kf}
    \htythy{\Sp_{\ge_0},\cS}.
  \]
  In each case, we call morphisms in $\cS$ \emph{stable equivalences}.  Let $\cSi \subset \GsCat$ and $\cSi \subset \GssSet$ denote the right-induced subcategories, created by
  \[
    i^* \cn \GsCat \to \GaCat \andspace
    i^* \cn \GssSet \to \GasSet.
  \]
  The morphisms in $\cSi$ are called \emph{$i^*$-stable equivalences}.
\end{definition}

Since $i\cn \Fskel \to \Gskel$ is fully faithful and the stable equivalences of connective spectra include all isomorphisms and satisfy 2-out-of-3, we have the following application of \cref{theorem:Li-adj-he}.
\begin{corollary}\label{corollary:Li-htyeq}
  The adjunctions $L \dashv i^*$ are equivalences of homotopy theories
  \begin{align*}
    L\cn \htythy{\GaCat, \cS} & \lrsimadj \htythy{\GsCat, \cSi}\cn i^*
    \andspace\\
    L\cn \htythy{\GasSet, \cS} & \lrsimadj \htythy{\GssSet, \cSi}\cn i^*.
  \end{align*} 
\end{corollary}

Now we use \cref{corollary:Li-htyeq} to prove our main results about equivalences of homotopy theories in \cref{eq:summary}.

\begin{theorem}\label{theorem:Jem-sma-N-hty-eq}
  The following are equivalences of homotopy theories:
  \begin{equation}
    \sma^* \cn \htythy{\GaCat, \cS}
    \fto{\sim} \htythy{\GsCat, \cSi},\label{it:smastar-hty-eq}
  \end{equation}
  \begin{equation}
    \sma^* \cn \htythy{\GasSet, \cS}
    \fto{\sim} \htythy{\GssSet, \cSi},\label{it:smastarsset-hty-eq}
  \end{equation}
  \begin{equation}
    \qquad
    \Jem \cn \htythy{\permcatsu, \cS}
    \fto{\sim} \htythy{\GsCat, \cSi},\andspace\label{it:Jem-hty-eq}
  \end{equation}
  \begin{equation}
    N_* \cn \htythy{\GsCat,\cSi}
    \fto{\sim} \htythy{\GssSet,\cSi}.\label{it:Nstar-hty-eq}
  \end{equation}
\end{theorem}
\begin{proof}
  Consider the following.
  \[
    \begin{tikzpicture}[x=20mm,y=13mm,scale=1.6,vcenter]
      \draw[0cell] %objects
      (0,-.5) node (a) {\htythy{\permcatsu,\cS}}
      (1,0) node (b) {\htythy{\GaCat,\cS}}
      (1,-1) node (c) {\htythy{\GsCat,\cSi}}
      ;
      \draw[1cell] %arrows
      (a) edge node {\Jse} node['] {\sim} (b)
      (a) edge node['] {\Jem} (c)
      (c) edge[',transform canvas={shift={(-.1,0)}}] node {i^*} node['] {\sim} (b)
      (b) edge[transform canvas={shift={(.15,0)}}] node {\sma^*} (c)
      ;
    \end{tikzpicture}
  \]
  The functors $\Jse$ and $i^*$ create stable, respectively $i^*$-stable, equivalences, and therefore are relative functors.
  Work of Thomason \cite{thomason} shows that $\Jse$ is an equivalence of homotopy theories.
  An independent proof by Mandell \cite{mandell_inverseK} constructs an inverse directly.
  \cref{corollary:Li-htyeq} shows that $i^*$ is an equivalence of homotopy theories.

  For \cref{it:smastar-hty-eq}, note that we have
  \[
    \sma \circ i = 1 \cn \Fskel \to \Fskel
  \]
  and hence $i^* \sma^* = 1$, the identity on $\GaCat$.
  This implies that $\sma^*$ is a relative functor and an equivalence of homotopy theories by \cref{lemma:rel-2-of-3} with $F' = i^*$ and $F = \sma^*$.  The proof for \cref{it:smastarsset-hty-eq} is the same by replacing $\Cat$ with $\sSet$.

  For \cref{it:Jem-hty-eq}, a similar reduction for smash products over singletons shows that $\Jse = i^*\Jem$ by unpacking the definitions \cite[10.3.1,10.3.27]{cerberusIII}.
  This implies that $\Jem$ is a relative functor and an equivalence of homotopy theories by \cref{lemma:rel-2-of-3} with $F' = i^*$ and $F = \Jem$. 
  
  For \cref{it:Nstar-hty-eq}, consider the following diagram, which commutes by associativity of functor composition.
  \[
    \begin{tikzpicture}[x=40mm,y=15mm]
      \draw[0cell] %objects
      (0,0) node (a) {\htythy{\GaCat,\cS}}
      (1,0) node (b) {\htythy{\GasSet,\cS}}
      (0,-1) node (c) {\htythy{\GsCat,\cSi}}
      (1,-1) node (d) {\htythy{\GssSet,\cSi}}
      ;
      \draw[1cell] %arrows
      (a) edge node {N_*} node[']{\sim} (b)
      (c) edge node {N_*} (d)
      (c) edge node {i^*} node[']{\sim} (a)
      (d) edge['] node {i^*} node[']{\sim} (b)
      ;
    \end{tikzpicture}
  \]
  The two instances of $i^*$ create $i^*$-stable equivalences and are equivalences of homotopy theories by \cref{corollary:Li-htyeq}.
  The top functor $N_*$ creates the stable equivalences of $\Ga$-categories by definition.
  It induces an equivalence of homotopy theories by work of Mandell \cite{mandell_inverseK}.
  Therefore, by commutativity of the diagram and \cref{lemma:rel-2-of-3} again, the lower instance of $N_*$ is a relative functor and an equivalence of homotopy theories.
\end{proof}

Next we recall the following result due to Elmendorf-Mandell \cite{elmendorf-mandell,cerberusIII}.
We will use this to prove that $\Kem$ is an equivalence of homotopy theories.

\begin{theorem}[{\cite[10.6.10]{cerberusIII}}]\label{theorem:KgNPistar}
  The natural transformation $\Pi^*$ in \cref{eq:summary} induces a componentwise stable equivalence from Segal $K$-theory, $\Kse$, to Elmendorf-Mandell $K$-theory, $\Kem$.  That is,
  \[
    \Kg N_* \Pi^* \cn \Kf N_* \Jse \C = \Kg N_* \sma^* \Jse \C \to \Kg N_* \Jem \C
  \]
  is a stable equivalence of connective spectra for each small permutative category $\C$.
\end{theorem}

The functor $\Kem$ is the composite
\[
  \Kem = \Kg N_* \Jem.
\]
While
\[
  \Jem \cn \htythy{\permcatsu,\cS} \to \htythy{\GsCat,\cSi}
  \text{\ \ and\ \ }
  N_* \cn \htythy{\GsCat,\cSi} \to \htythy{\GssSet,\cSi}
\]
are relative functors, $\Kg$ may not be.  The reason is that, in order for $\Kg$ to be a relative functor, it would need to send an $i^*$-stable equivalence $f$ in $\GssSet$ to a stable equivalence $\Kg(f)$ between symmetric spectra.  However, we only know that $i^*(f)$ is a stable equivalence in $\GasSet$, which, in turn, means that $\Kf i^*(f)$ is a stable equivalence between symmetric spectra.  While there are equalities 
\[\Kg \sma^* = \Kf \andspace i^* \sma^* = 1_{\GasSet},\] 
we do not know how $\Kg$ is related to $\Kf i^*$.  Thus the first step in proving the following result is to use naturality of $\Pi^*$ to show that the composite $\Kem$ \emph{is} a relative functor.
\begin{theorem}\label{theorem:Kem-heq}
  The Elmendorf-Mandell $K$-theory functor is an equivalence of homotopy theories
  \[
    \Kem\cn \htythy{\permcatsu, \cS} \fto{\sim} \htythy{\Sp_{\ge 0}, \cS}.
  \]
\end{theorem}
\begin{proof}
  For a strictly unital symmetric monoidal functor
  \[
    F \cn \C \to \D \inspace \permcatsu
  \]
  we have the following naturality square.
  \[
    \begin{tikzpicture}[x=40mm,y=15mm]
      \draw[0cell] %objects
      (0,0) node (a) {\Kse(\C)}
      (0,-1) node (c) {\Kse(\D)}
      (1,0) node (b) {\Kem(\C)}
      (1,-1) node (d) {\Kem(\D)}
      ;
      \draw[1cell] %arrows
      (a) edge node {\Kg N_* (\Pi^*_{\C})} (b)
      (c) edge node {\Kg N_* (\Pi^*_{\D})} (d)
      (a) edge['] node {\Kse(F)} (c)
      (b) edge node {\Kem(F)} (d)
      ;
    \end{tikzpicture}
  \]
  By \cref{theorem:KgNPistar}, the components of  
  \[
    p = \Kg N_* \Pi^* \cn \Kse \to \Kem
  \]
  are stable equivalences of connective spectra.
  If $F$ is a stable equivalence in $\permcatsu$, then $\Kse(F)$, and therefore also $\Kem(F)$, are stable equivalences.

  This shows that $\Kem$ is a relative functor.
  The result then follows from \cref{lemma:hteq-transf} because $\Kse$ is an equivalence of homotopy theories \cite{segal,bousfield-friedlander}.
\end{proof}

\begin{remark}\label{remark:pistar}
We note that $\Pi^*$ is a componentwise $i^*$-stable equivalence.
  It follows from the definition of $\Pi$ in \cite[10.6.1]{cerberusIII} that $i^* \Pi^*_\C$ is the identity morphism of the $\Ga$-category $\Jse\C$, for each small permutative category $\C$.
  Therefore, the components of $\Pi^*$ are $i^*$-stable equivalences of $\Gstar$-categories.  We use this fact in \cref{pistarnatiso} below.
\end{remark}

\subsection*{Application to Bohmann-Osorno $K$-Theory}

Work of Bohmann-Osorno \cite{bohmann_osorno} gives a different multifunctorial $K$-theory construction, factoring through diagrams on a category $\cE$ that is similar to $\Gskel$, but uses the opposite simplex category, $\Deltaop$, in place of $\Fskel$ in \cref{definition:Gstar} above.
Note that this description of $\cE$ is different from that of \cite[5.1]{bohmann_osorno}, but pointed diagrams on the category $\cE$ as described here are the $\cE_*$-diagrams of \cite[5.4]{bohmann_osorno}.
There is a fully faithful inclusion of length-one tuples
\[
  j\cn \Deltaop \to \cE
\]
and hence \cref{theorem:Li-adj-he} gives the following.  A similar result for $\Deltaop_*\mh\sSet$ also holds.
\begin{corollary}\label{corollary:BO-htyeq}
  Suppose $\cS \subset \Deltaop_*\mh\Cat$ is a subcategory containing all objects, and suppose, moreover, that $\cS$ contains all isomorphisms and has the 2-out-of-3 property.
  Let $\cS^j = (j^*)^\inv\cS$ be the right-induced subcategory of $\cE_*\mh\Cat$.
  Then there is a left adjoint $L \dashv j^*$ inducing equivalences of homotopy theories
  \begin{align*}
    L\cn \htythy{\Deltaop_*\mh\Cat, \cS} & \lrsimadj \htythy{\cE_*\mh\Cat, \cS^j}\cn j^*.
  \end{align*} 
\end{corollary}
\noindent Thus, the right-induced homotopy theory of $\cE_*$-categories is equivalent to the homotopy theory of pointed simplicial categories.

\section{Homotopy Inverses of \texorpdfstring{$K$}{K}-Theory Multifunctors}\label{sec:hinv}

In this section we identify explicit homotopy inverses of the $K$-theory multifunctors $\Jem$, $N_* \Jem$, and $\Kem$, each of which is an equivalence of homotopy theories by \cref{theorem:Jem-sma-N-hty-eq,theorem:Kem-heq}.  These homotopy inverses are defined in, respectively, \cref{eq:jeminv,eq:njeminv,eq:keminv}.  The following is a summary diagram.
\begin{equation}\label{eq:invsummary}
\begin{tikzpicture}[x=20mm,y=13mm,scale=1.6,vcenter]
    \draw[0cell] %objects
    (0,0) node (a) {\permcatsu}
    (a)+(30:1.1) node (b) {\GaCat}
    (a)+(-30:1.1) node (b') {\GsCat}
    (b)+(1,0) node (c) {\GasSet}
    (b')+(1,0) node (c') {\GssSet}
    (c)+(-30:1.1) node (d) {\Sp_{\ge 0}}
    ;
    \draw[1cell] %arrows
    (a) edge node {\Jse} (b)
    (a) edge['] node {\Jem} (b')
    (b) edge node {N_*} (c)
    (b') edge node {N_*} (c')
    (c) edge node {\Kf} (d)
    (c') edge['] node {\Kg} (d)
    (b) edge[',transform canvas={shift={(-.1,0)}}] node[pos=.4] {\sma^*} (b')
    (b') edge node[',pos=.5] {i^*} (b)
    (c) edge[transform canvas={shift={(.1,0)}}] node[pos=.4] {\sma^*} (c')
    (c') edge[transform canvas={shift={(0,0)}}] node[pos=.5] {i^*} (c)
    ;
    \draw[1cell]
    (d) edge[bend right=40] node[swap] {\bA} (c)
    (c) edge[bend right] node[swap] {S_*} (b)
    (b) edge[bend right=40] node[swap] {\cP} (a)
    ;
    \draw[1cell]
    (a) [rounded corners=10pt] |- ($(b')+(0,-.35)$)
    -- node {\Kem} ($(c')+(0,-.35)$) -| (d)
    ;
    \draw[2cell]
    (a) +(-5:.55) node[rotate=-120, 2label={below,\Pi^*\!\!}] {\Rightarrow}
    ;
  \end{tikzpicture}
\end{equation}\ \\
The functors along the top, $\cP$, $S_*$, and $\bA$, are homotopy inverses of, respectively, $\Jse$, $N_*$, and $\Kf$.  The other functors are as in \cref{eq:summary}.

\subsection*{Homotopy Inverse of $\Jem$}

By \cref{remark:pistar} 
\begin{equation}\label{pistarnatiso}
\Pi^* \cn \sma^* \Jse \to \Jem \cn \htythy{\permcatsu, \cS} \to \htythy{\GsCat, \cSi}
\end{equation}
induces a natural isomorphism in the stable homotopy categories, obtained by inverting the respective classes of stable equivalences.  By \cref{corollary:Li-htyeq,theorem:Jem-sma-N-hty-eq}, both relative functors
\begin{equation}\label{smastaristar}
\begin{tikzcd}[column sep=large]
\htythy{\GaCat, \cS} \ar{r}{\sma^*}[swap]{\sim} 
& \htythy{\GsCat, \cSi} \ar{r}{i^*}[swap]{\sim} & \htythy{\GaCat, \cS}
\end{tikzcd}
\end{equation}
are equivalences of homotopy theories, and the analogous statement holds with $\sSet$ in place of $\Cat$.  The equality $i^* \sma^* = 1$ implies that the functors $\sma^*$ and $i^*$ are \emph{homotopy inverses} of each other.  This means that, passing to the stable homotopy categories, they induce mutually inverse equivalences of categories.  

A homotopy inverse of Segal $J$-theory, $\Jse$, is given by Mandell's inverse $K$-theory functor \cite[Theorem 4.5]{mandell_inverseK}
\begin{equation}\label{mandellP}
\cP \cn \htythy{\GaCat, \cS} \fto{\sim} \htythy{\permcatsu, \cS}.
\end{equation}
This is a non-symmetric $\Cat$-enriched multifunctor by \cite[Theorem 1.3]{johnson-yau-invK}.  Combining \cref{pistarnatiso,smastaristar,mandellP}, a homotopy inverse of Elmendorf-Mandell $J$-theory, $\Jem$, is given by the following composite.
\begin{equation}\label{eq:jeminv}
\begin{tikzcd}[column sep=large]
\htythy{\GsCat, \cSi} \ar{r}{i^*}[swap]{\sim} & \htythy{\GaCat, \cS} \ar{r}{\cP}[swap]{\sim} & \htythy{\permcatsu, \cS}
\end{tikzcd}
\end{equation}

We note that---while $\Jem$ and $\cP$ are $\Cat$-enriched multifunctors (non-symmetric in the case of $\cP$)---the functor $i^*$ does \emph{not} appear to be a multifunctor, even in the non-symmetric sense.  In particular, $i^*$ does not admit any reasonable monoidal constraint that is compatible with the monoidal structures on $\GaCat$ and $\GsCat$.

\subsection*{Homotopy Inverse of $N_* \Jem$}

A homotopy inverse of 
\[N_* \cn \htythy{\GaCat, \cS} \fto{\sim} \htythy{\GasSet, \cS}\]
is given by Mandell's functor $S_*$ in \cite[Proposition 2.7]{mandell_inverseK}, denoted $\cS$ there.  This is \emph{not} a multifunctor, as discussed in the introduction of \cite{johnson-yau-invK}.  
By \cref{pistarnatiso} and the equality 
\[N_* \sma^* \Jse = \sma^* N_* \Jse,\]
a homotopy inverse of the composite
\[\begin{tikzcd}[column sep=large]
\htythy{\permcatsu, \cS} \ar{r}{\Jem}[swap]{\sim} 
& \htythy{\GsCat, \cSi} \ar{r}{N_*}[swap]{\sim}
& \htythy{\GssSet, \cSi}
\end{tikzcd}\]
is given by the following composite.
\begin{equation}\label{eq:njeminv}
\begin{tikzcd}[column sep=normal]
\htythy{\GssSet, \cSi} \ar{r}{i^*}[swap]{\sim}
& \htythy{\GasSet, \cS} \ar{r}{S_*}[swap]{\sim} 
& \htythy{\GaCat, \cS} \ar{r}{\cP}[swap]{\sim} & \htythy{\permcatsu, \cS}
\end{tikzcd}
\end{equation}

\subsection*{Homotopy Inverse of $\Kem$}

A homotopy inverse of
\[\Kf \cn \htythy{\GasSet, \cS} \fto{\sim} \htythy{\Sp_{\ge_0},\cS}\]
is given by Segal's functor $\bA$ in \cite[Definition 3.1]{segal}.  The adjoint pair $\Kf \dashv \bA$ is a Quillen equivalence by \cite[Theorem 5.8]{bousfield-friedlander}.  By \cref{pistarnatiso} and the equalities 
\[\Kg N_* \sma^* \Jse = \Kg \sma^* N_* \Jse = \Kf N_* \Jse,\] 
a homotopy inverse of Elmendorf-Mandell $K$-theory, $\Kem = \Kg N_* \Jem$, is given by the following composite along the top of \cref{eq:invsummary}.
\begin{equation}\label{eq:keminv}
\begin{tikzcd}[column sep=normal]
\htythy{\Sp_{\ge_0},\cS} \ar{r}{\bA}[swap]{\sim} 
& \htythy{\GasSet, \cS} \ar{r}{S_*}[swap]{\sim} 
& \htythy{\GaCat, \cS} \ar{r}{\cP}[swap]{\sim} & \htythy{\permcatsu, \cS}
\end{tikzcd}
\end{equation}

We note that---while $\Kem$ and $\cP$ are enriched multifunctors (non-symmetric in the case of $\cP$)---the functor $\bA$ does \emph{not} appear to be a multifunctor, even in the non-symmetric sense.  Similar to $i^*$, the functor $\bA$ does not admit any reasonable monoidal constraint that is compatible with the monoidal structures on $\GasSet$ and $\Sp_{\geq 0}$.

% amsalpha3 is an enhanced version of amsalpha
\bibliographystyle{sty/amsalpha3}
\bibliography{references}
%index style sty/idx-style.ist loaded from Packages.sty
%\printindex
\end{document}